\documentclass[a4paper,10pt]{article}
\usepackage{amsfonts,amsmath,amssymb,amsthm}
\usepackage[english]{babel}
\usepackage[T1]{fontenc}
\usepackage[ansinew]{inputenc}  

\newcommand{\Naturel}{\mathbb N}
\newcommand{\Relatif}{\mathbb Z}
\newcommand{\Rationnel}{\mathbb Q}
\newcommand{\Reel}{\mathbb R}

\newcommand{\HR}{{^*\Reel}}

\newcommand{\Hal}{\text{\rm Hal}}
\newcommand{\Gal}{\text{\rm Gal}}
\newcommand{\KK}{\mathbb K}

\newcommand{\Eps}{\varepsilon}

\newtheorem{definition}{Definition}
\newtheorem{theoreme}{Theorem}
\newtheorem{corollaire}{Corollary}
\newtheorem{proposition}{Proposition}

\begin{document}

\title{\textbf{Contraction of a Generalized\\ Metric Structure}}
\author{Guy Wallet}

\maketitle

\begin{center}
{\small Laboratoire de Mathématiques, Image et Applications \\
Pôle Sciences et Technologie de l'Université de La Rochelle \\ Avenue
Michel Crépeau \hspace{3mm}
17042 La Rochelle cedex 1 \\ guy.wallet@univ-lr.fr }
\end{center}

\textbf{Abstract}: {\small 
In some scientific fields, a scaling is able to modify the topology of an observed object. Our goal in the present work is to introduce a new formalism adapted to the mathematical representation of this kind of phenomenon. To this end, we introduce a new metric structure - the galactic spaces - which depends on an ordered field extension of $\Reel$. Moreover, some natural transformations of the category of galactic spaces, the contractions, are of particular interest: they generalize usual homotheties, they have a ratio which may be an infinitesimal, they are able to modify the topology and they satisfy a nice composition rule. With the help of nonstandard extensions we can associate to any metric space an infinite family of galactic spaces; lastly, we study some limit properties of this family.  }
 
\vskip 2mm



\section{Introduction: scaling and topology}

It is well known that the notion of scale is fundamental in empirical sciences: the properties of an object generally depend on a given scale and a change of scale (a scaling) may deeply modify some of these properties. In some fields like Image Processing, Geographical Information Systems and Spatial Analysis, one of the major effects of a scaling is a possible modification of the topology. An elementary example: a city $A$ which is inside a geographical area $B$ at some scale may be located on the boundary of $B$ when considered at a smaller scale. A general question is to be able to take into account these topological deformation phenomena \cite{J99,JB95,GW07}. For instance, it is a real problem to identify a given object represented at different scales. Let us notice that the scalings for which we hope a determinist law are the \emph{contractions}, i.e. the scalings which decrease  the scale (and the size of the objects).

The mathematical transformation naturally related to a scaling is the notion of homothety (for instance in an affine space). But an homothety is an homeomorphism; thus, it leaves invariant the topology. On the basis of this observation, many experts in Geographical Information Systems and Spatial Analysis concluded that a scaling is a natural transformation which cannot be exactly represented by a mathematical transformation. Although this opinion is the expression of a real difficulty, it underestimates the modeling capacity of mathematics.

In order to progress in our analysis, we must introduce a general mathematical framework  adapted to this kind of problem. For that purpose, we consider the class $\mathcal{E}$ of all metric spaces. Given a real number $\lambda >0$, the homothety of ratio $\lambda$ is the quit trivial transformation operating on $\mathcal{E}$ which maps each metric space $E:=(E,d)$ to the new one $\lambda E:=(E,\lambda d)$. We notice that we have the nice composition property $\lambda(\mu E)=(\lambda \mu)E$ and more generally, that the homotheties result from the left action$(\lambda,E) \mapsto \lambda E$ of the multiplicative group $\Reel_+^*$ on $\mathcal{E}$. Nevertheless, since the distance $d$ and $\lambda d$ define the same topology on the set $E$, we find again the invariance property of topology by an homothety. Thus, it is still true that an homothety is not a good representative of a scaling. 

On the other side, we can also agree that an empirical scaling is much more that a simple change of size. In reality, \emph{a concrete scaling seems to be the union of two distinct but dependent processes: (1) an homothety which changes, possibly very strongly, the size of any object, (2) a simplification which allows to neglect too small details}. In order to build a convenient mathematical concept of scaling, we have to translate simultaneously these two processes in an appropriate mathematical notion. Actually, a major work was already done on this topic. It is about the limit of a sequence $(\lambda_n E)$ in $\mathcal{E}$ for the Gromov-Hausdorff distance where $E$ is a given metric space and $(\lambda_n)$ is a sequence in $\Reel_+^*$ such that $\lambda_n \rightarrow 0$. Introduced by Gromov in his study on group of polynomial growth in 1981 \cite{Grom81,GromLafPans,Grom99}, this concept collects the two main aspects of a concrete scaling: the sequence $(\lambda_n)$ corresponds to the strong homothety (actually a strong contraction) and the limit corresponds to the simplification process. Nevertheless, it is not easy to handle with this kind of limit because there are few general results of convergence for a  sequence of the type $(\lambda_n E)$. In 1984 \cite{vdDW}, Van den Dries and Wilkie defined a non standard alternative: the asymptotic cone of a metric space $E$ relative to a center $x_0$ in a nonstandard extension ${^*\!E}$ of $E$ and for a similitude ratio $\lambda$ which is now an infinitesimal hyperreal number. The main advantage of this construction is that its existence is certain, whatever be $E$, $\lambda$ and $x_0$; its strongest disadvantage is its transcendent character due to its dependance to a non trivial ultrafilter.   Since their introduction, these concepts have been the subject of many deep and interesting works at the border of group theory, topology and logic (for instance \cite{Drutu01,DruSap05,KSTT,EO05}). Nevertheless, we notice that, in the framework of this two approaches, it is difficult to formulate and give a simple meaning to a composition rule of two scaling. 

Our study is not strictly in the field of the preceding works on asymptotic cones. Indeed, our main goal is to bring a positive answer to the following question. \emph{Is it possible to  generalize the notion of homothety so as to get a class $\mathcal{S}$ of transformations operating on a class $\mathcal{G}$ of spaces with the properties set out below?
\begin{enumerate}
\item Each space $G \in \mathcal{G}$ is provided with a kind of metric structure which is a generalization of the structure of a metric space.
\item Each transformation $s \in \mathcal{S}$ modify (possibly very strongly) the size of any object and the strength of this modification is measured by a ration $\lambda_s$.
\item The underlying topological structure   is possibly altered by such operations.
\item For each $s,s' \in \mathcal{S}$, there is the nice composition rule  $\lambda_{s \circ s'}=\lambda_s . \lambda_{s'}$.
\end{enumerate}}
The base of the present work is the observation that the construction of a nonstandard asymptotic cone actually produces an intermediary space which carries a more general metric structure than the structure of metric space. We call this new structure a galactic space. A remarkable feature is that the notion of galactic space is not strictly dependent of the nonstandard framework: to define it, we only need an ordered field extension of $\Reel$. Moreover, some natural transformations of the category of galactic spaces, the contractions, are of particular interest: they generalize usual homotheties but they have a ratio which may be an infinitesimal, they carry out a simplification process and they satisfy a nice composition rule. Thus, at the level of galactic spaces, the contractions seem to be good representatives of scalings. 


It seems probable that partially similar structures have already been used in other contexts. For instance, this is the case for the basic notion of distance taking its values in a ordered quotient group. It is only at the time he was finishing the final bibliography of this paper that the author found in \cite{KT04} the notion of ultracone which is almost an example of the general structure of galactic space.

\section{Infinitesimals in an extension field of $\Reel$}

\textbf{2.1} In all this study, we consider a \emph{proper ordered field
extension} $\KK$ of the field $\Reel$ of real numbers. Thus, $\KK$ is an ordered field, $\Reel \subsetneq \KK$ and the restriction to $\Reel$ of the ordered field structure of $\KK$ is the usual ordered field structure or $\Reel$.

We may think to some algebraic examples such as the field $\Reel(X)$
of rational fractions
$$\Reel(X) := \left\{\frac{P(X)}{Q(X)} \ ; \ (P(X),Q(X)) \in
\Reel[X]^2 \ \text{and} \ Q(X) \neq 0\right\}$$
or the field $\Reel((X))$ of Laurent power series
$$\Reel((X)) := \left\{ \sum_{i=m}^{+\infty}a_iX^i \ ; \ m \in
\Relatif \ \text{and} \ \forall i \geq m \ a_i \in \Reel \right\}$$
or the field of Puiseux series with real coefficients 
$$\Reel[[X^\Rationnel]]:=\bigcup_{n \geq 1} \Reel((X^{1/n}))$$
both provided with the order relation for which $0<X<1/n$ for all $n \in \Naturel^*$.

Another intersting example arises from nonstandard analysis: a
field $\HR$ or hyperreal numbers \cite{ANSLoeb,R66}. The simplest way to construct a field of hyperreal numbers is to take an ultra-power of $\Reel$.
To this end, we choose a non principal ultrafilter $\mathcal{U}$ on the
set $\Naturel$ of natural numbers (0 included). That means that
$\mathcal{U}$ is a family of subsets of $\Naturel$ such that
\begin{enumerate}
\item $\mathcal{U} \neq \emptyset$ and $\emptyset \not\in
\mathcal{U}$,
\item $\forall (U,V) \in \mathcal{P}(\Naturel) \
 (U \in \mathcal{U} \ \text{and} \ U \subset V \ \Longrightarrow
 \ V \in \mathcal{U})$,
\item $\forall (U,V) \in \mathcal{U} \ \ U \cap V \in \mathcal{U}$,
\item $\forall F \in \mathcal{P}(\Naturel) \ ( F \ \text{finite} \
\Longrightarrow F \not\in \mathcal{U})$.
\end{enumerate}
Then, we consider the set $\Reel^{\Naturel}$  of sequences of real
numbers and the equivalence relation $\sim_{\mathcal{U}}$ on
$\Reel^{\Naturel}$ such that $(x_n) \sim_{\mathcal{U}} (y_n)$ if and
only if $\{n \in \Naturel \ ; \ x_n=y_n\} \in \mathcal{U}$. The set
$\HR$ of hyperreal numbers is the quotient set
$\Reel^{\Naturel}/\sim_{\mathcal{U}}$. It is easy to check directly
(whithout using any special logical tool) that $\HR$ is an ordered
field for the natural operations and the order relation
$$[(x_n)]_{\mathcal{U}} \leq [(y_n)]_{\mathcal{U}} \
\Longleftrightarrow \ \{n \in \Naturel \ ; \ x_n \leq y_n \} \in
\mathcal{U}$$ where $[(x_n)]_{\mathcal{U}}$ and
$[(y_n)]_{\mathcal{U}}$ are the equivalence classes of the sequences
$(x_n)$ and $(y_n)$. The map $*:\Reel \rightarrow \HR$ such that the
image of any $x \in \Reel$ is the equivalence class of the sequence
of constant value $x$ is clearly a field morphism which preserves order relations. Thus, $\HR$ is an ordered field extension of
$\Reel$. Moreover, the class of the sequence
$(0,1,2,\ldots,n,\ldots)$ does not belong to the image of $\Reel$ in
$\HR$. Consequently, this extension is proper.

Now, we consider the general case of a proper ordered extension $\KK$ or $\Reel$.
Let $s$ an element of $\KK$; then $s$ is \emph{infinitely small} and
we write $s \simeq 0$ if $|s| < \frac{1}{n}$ for each $n \in
\Naturel^*$; $s$ is \emph{infinitely large} if $n < |s|$ for each $n
\in \Naturel$ (we write $s \simeq +\infty$ if $s>0$ and $s \simeq
-\infty$ if $s<0$); $s$ is \emph{limited} if it is not infinitely
large. We see that the inverse of an infinitely large element is an
infinitely small element.

\begin{proposition} 
 In $\KK$, there are infinitely large elements and non null infinitely small
 elements. Moreover, for each $s \in \KK$ which is not infinitely large, there
 exists one and only one ${^o\!s} \in \Reel$ such that $s \simeq
 {^o\!s}$.
\end{proposition}

\begin{proof}
Let $s \in \KK \setminus \Reel$. If $s$ is infinitely large, the
first point is proven. If not, we are under the assumptions of the
second point. Then, we consider the cut $(I,S)$ of $\Reel$ defined
by $I=\{x \in \Reel \ ; \ x<s \}$ and $S=\{x \in \Reel \ ; \ x>s \}$.
This cut defines one and only one real number ${^o\!s}$ such that
${^o\!s}=\sup I$ and ${^o\!s}=\inf S$. Then it is necessary that $s
\simeq {^o\!s}$ and we see that $s-{^o\!s}$ is infinitely small and
different from 0.
\end{proof}

\textbf{2.2} The \emph{halo} of 0 is the set $\Hal(0)$ whose elements are the
infinitely small elements of $\KK$. It is an additive subgroup of
$\KK$ to which is associated the proximity relation $\simeq$ defined
by
$$\forall (s,t) \in \KK^2 \ \left( s \simeq t \
\Longleftrightarrow s-t \in \Hal(0)\right)$$ The halo of any $x \in
\KK$ is the equivalence class $\Hal(x)$ of $x$ for $\simeq$, that is
to say $\Hal(x)=x + \Hal(0)$. The quotient group $\Hal(\KK):=\KK / \Hal(0)$ is the  set of all  $\Hal(x)$ for $x \in \KK$ and the map $\Hal : x \mapsto \Hal(x)$ is the canonical projection of $\KK$ on $\Hal(\KK)$.

The \emph{galaxy} of 0 is the set $\Gal(0)$ whose elements are the
limited elements of $\KK$.

It is clear that $\Hal(0) \subset \Gal(0)$ and that $\Gal(0)$ is
also an additive subgroup of $\KK$. From the preceding proposition,
we deduce that $\text{Gal}(0)/\text{Hal}(0)$ is isomorphic to
$\Reel$ and $\Gal(0)$ is equal to the disjoint union $\displaystyle
\bigcup_{x \in \Reel} \Hal(x)$. Moreover,  there exists a
\emph{principal value map} $\text{pv}:\KK \rightarrow \Reel \cup
\{+\infty,-\infty\}$ such that
$$\forall s \in \KK \ \ \text{pv}(s)=
 \left\{
  \begin{array}{ll}
    {^o\!s} & \text{if $s \in \Gal(0)$ (where ${^o\!s} \in \Reel$ and ${^o\!s} \simeq s$) } \\
    +\infty & \text{if $s \simeq +\infty$} \\
    -\infty & \text{if $s \simeq -\infty$} \\
  \end{array}
 \right.
$$
so that $s \simeq \text{pv}(s)$ for every $s \in \KK$. If $\KK$ is a
field of nonstandard hyperreal numbers ${^*\Reel}$, the map
$\text{pv}$ is usually called the \emph{standard part map} and
denoted $\text{st}$. Due to the fact that the sum of two
infinitesimals is an infinitesimal, we see that
$\text{pv}(s+t)=\text{pv}(s)+\text{pv}(t)$ for every $t,s \in \KK$
such that $s,t \not\simeq \pm \infty$.

Let us consider the quotient group $\Gal(\KK):=\KK/\text{Gal}(0)$
and the canonical projection $\text{Gal}:\KK \rightarrow \Gal(\KK)$.
For each $t \in \KK$, the equivalence class $\Gal(t)=t+\Gal(0)$ is
called the galaxie of $t$.

There is a natural total order relation $\leq$ on $\Gal(\KK)$ defined by
$$\forall (s,t) \in \KK^2 \ \left[ \Gal(s) \leq
\Gal(t) \ \Longleftrightarrow \ (s \leq t \ \text{or} \ s - t \in
\Gal(0))\right]$$
Moreover, this relation is compatible with the additive structure of
$\Gal(\KK)$
$$\forall (s,t) \in \KK^2 \ \left[ (0 \leq \Gal(s) \
\text{and} \ 0 \leq \Gal(t)) \ \Longrightarrow \ 0 \leq
\Gal(s)+\Gal(t)\right]$$ 
These properties mean that $\Gal(\KK)$ is
an ordered additive group.

The terms of halo and galaxy are already used in some development of nonstandard analysis; in this case, they denotes two important classes of external sets \cite{DR89,vdB87}. Our study is not directly in connection with these specific properties. In addition, it is probable there are similar concepts in other contexts using non-archimedean extensions.

\section{Galactic space}

Given any ordered additive group $G$, we denote by $G_+$ and  $G_+^*$  the following  sets $G_+=\{x \in G \ ; \ x \geq 0\}$ and $G_+^*=\{x \in G \ ;  \ x>0\}$. The symbol $+\infty$ is supposed such that, for any $x \in G$, we have $x<+\infty $ and $x+(+\infty)=+\infty+x=+\infty$. Given a map $d$ defined on a product $X \times X$ and taking its values in $G_+$ or $G_+ \cup \{+\infty\}$, we say that $d$ satisfies the \emph{general metric rule} if , for every $x,y,z \in X$
\begin{enumerate}
\item $d(x,x)=0$
\item $d(x,y)=d(y,x)>0$ for $x \neq y$
\item $d(x,z) \leq d(x,y)+d(y,z)$
\end{enumerate}

Hence, a metric space is a structure  $(X,d)$ such that $X$ is a set and $d$ is a  a map $d:X \times X \rightarrow \Reel_+$ which satisfies the general metric rule (i.e. $d$ is a distance on $X$). 

\begin{definition}  Given a set $F$, a map $\delta:F \times F \rightarrow \Reel_+ \cup \{+\infty\}$ which satisfies the general metric rule is called a \emph{generalized distance} on $F$.
\end{definition}

If $\delta$ is a generalized distance on $F$, for each $(x,r) \in F \times (\Reel_+ \cup \{+\infty\})$, we can define  the open ball of center $x$ and radius $r$
$$B_\delta(x,r)=\{y \in F \ ; \ \delta(x,y)<r\}$$
The open balls of radius $+\infty$ are also the equivalence class for the relation
$\delta(x,y)<+\infty$. Each large open ball is clearly a metric
space for $\delta$ and is called \emph{a metric component of $F$ for
$\delta$}. Let $\mathcal{M}_F$ be the set of all metric components
of $F$ for $\delta$ and, for each $x \in F$, let $C_F(x)$ the
element $E \in \mathcal{M}_F$ such that $x \in E$.

If $\delta$ is a generalized distance on a set $F$,
then the family of open balls is a basis of a topology on $F$.
Furthermore, $F$ is the disjoint union of its metric components and
each metric component is an open set of $F$.
Reciprocally, if $(E_i,\delta_i)_{i \in I}$ is a family of disjoints
metric spaces, then $(E_i)_{i \in I}$ is the family of metric
components of $F=\bigcup_{i \in I}E_i$ for the generalized distance
$\delta$ defined by
$$\forall (x,y) \in F^2 \ \ \delta(x,y)=
\left\{\begin{array}{cl}
         \delta_i(x,y) & \text{if $\exists i \in I$ such that $x,y \in E_i$,} \\
         +\infty & \text{else.} \\
       \end{array}
\right.$$

Consequently, a set provided with a generalized distance is just the
disjoint union of a family of metric spaces. We want to improve this
concept by the consideration of a kind of metric on the set of
metric components.

Let $\KK$ be a fixed ordered extension field of $\Reel$.
Therefore, we have the ordered group $\Gal(\KK)$ whose elements are
the galaxies of $\KK$.

\begin{definition}
A \emph{galactic distance} on a set $\mathcal{E}$ is a map
$\Delta:\mathcal{E} \times \mathcal{E} \rightarrow \Gal(\KK)$ which satisfies the general metric rule.
\end{definition}

If $\Delta$ is a galactic distance on a set $\mathcal{E}$, there is
a well defined topology on $\mathcal{E}$ so that the family of open
balls is a basis of this topology.

\begin{definition}
A \emph{galactic space}  is a structure $(F,\delta,\Delta)$ in which
$F$ is a set, $\delta$ is a generalized distance on $F$ and $\Delta$
is a galactic distance on the set $\mathcal{M}_F$ of metric
components of $F$ for $\delta$.
\end{definition}

With the aim of simplifying the notations, we can also say that $F$ is a galactic space without mentioning $\delta$ and $\Delta$. In some way, a galactic space is a set with two levels of resolution: a fine resolution given by the generalized distance which relates the topological relations between points inside each metric component, a coarse resolution given by the galactic distance which relates the topological relations between the metric components. In spite of the chosen terminology, the reader must avoid to think that the structure of galactic space may have any application in the science of universe. 

\begin{definition}
An \emph{isometry} between two galactic spaces $(F,\delta,\Delta)$
and $(F',\delta',\Delta')$ is a bijective map $\phi:F \rightarrow
F'$ such that
\begin{enumerate}
\item $\forall (x,y) \in F^2 \ \
\delta'(\phi(x),\phi(y))=\delta(x,y)$
\item $\forall (G,H) \in \mathcal{M}_F^2 \ \
\Delta'(\phi(G),\phi(H))=\Delta(G,H)$
\end{enumerate}
\end{definition}

(It is clear that, if $\phi$ is bijective and satisfies the point (1.), then, for each metric component $G$ of $F$ for $\delta$ ,  the set $\phi(G)$ is a metric component of $F'$ for $\delta'$.)

\textbf{Example~1} A metric space $(E,d)$ is a particular case of
galactic space, that is to say the galactic space $(E,d,D)$ where
$D$ is the trivial galactic distance on the set $\{E\}$. This 
trivial example shows that the notion of galactic space is really a generalization of that of metric space.

\vskip 2mm

\textbf{Example~2} Let $(E_1,d_1)$ and $(E_2,d_2)$ two metric spaces
such that $E_1 \cap E_2=\emptyset$. We choose a galaxy $\Delta_{1,2}
\in \Gal(\KK)$ which is not trivial ($\Delta_{1,2} \neq 0 \in
\Gal(\KK)$). Then, it is easy to verify that there is a galactic
space $(E,d,D)$ such that
\begin{itemize}
\item $E=E_1 \cup E_2$~;
\item $\displaystyle d(x,y)=\left\{
 \begin{array}{ll}
   d_1(x,y) & \text{if $(x,y) \in {E_1}^2$} \\
   d_2(x,y) & \text{if $(x,y) \in {E_2}^2$} \\
   +\infty & \text{else} \\
 \end{array}\right.$~;
 \item the metric components of $E$ for $d$ are $E_1$ and $E_2$ and $D(E_1,E_2)=\Delta_{1,2}$.
\end{itemize}

\vskip 2mm

\textbf{Example~3} Let $(G,d)$ be a $\KK$\emph{-metric space}: that means
that $G$ is a set and $d$ is a map from $G \times G$ to $\KK_+$ which satisfies the general metric rule.
The simplest example of a $\KK$-metric space is $G=\KK$ and
$d(x,y)=|x-y|$ for each $(x,y) \in \KK^2$. Another example is
$G={^*\!E}$ and $d={^*\!d}$ where $(E,d)$ is a metric space,
$({^*\!E},{^*\!d})$ is a nonstandard extension of $(E,d)$ and
$\KK=\HR$.

Then, we consider the equivalence relation $\approx$ on $G$ defined
by
$$\forall (x,y) \in G^2 \ \ (x \approx y \Leftrightarrow d(x,y)
\simeq 0)$$ and the quotient set $F=G/\approx$. For every $x \in G$,
we denote $[x]$ the equivalence class of $x$ for $\approx$. On $F$
we have a generalized distance $\delta$ defined by
$$\forall (x,y) \in G^2 \ \ \delta([x],[y])=\text{pv}(d(x,y))$$

Finally, we define a galactic distance $\Delta$ on the set
$\mathcal{M}_F$ of metric components of $F$ for $\delta$ such that
$$\forall (x,y) \in G^2 \ \ \Delta(C_F(x),C_F(y))=\Gal(d(x,y))$$
where $C_F(t)$ denotes the metric component of $t \in F$. Then
$(F,\delta,\Delta)$ is a galactic space. We say that
$(F,\delta,\Delta)$ is the \emph{galactic projection} of the
$\KK$-metric space $(G,d)$.

\vskip 2mm

The following result shows that the preceding example is universal.

\begin{theoreme}
Every galactic space is the galactic projection of a $\KK$-metric
space.
\end{theoreme}

\begin{proof}
We consider a galactic space $(F,\delta,\Delta)$ and let
$\mathcal{M}_F$ be the set of metric components of this space. For
each $(E_i,E_j) \in \mathcal{M}_F^2$ such that $E_i \neq E_j$, we
choose an element $d_{ij}$ in the galaxy $\Delta(E_i,E_j)$ in such a
way that $d_{ij}=d_{ji}$. Thus, we have $d_{ij} \simeq +\infty$ in
$\KK$ and $\Gal(d_{ij})=\Delta(E_i,E_j)$. In the same way, we choose
an element $x_i$ in each metric component $E_i$. Then, we define a
map $d:F^2 \rightarrow \KK_+$ such that, for each $(x,y) \in F^2$
$$d(x,y)=
 \left\{
  \begin{array}{ll}
    \delta(x,y) & \text{if $\delta(x,y)<+\infty$} \\
    \delta(x,x_i)+d_{ij}+\delta(x_j,y) &
    \text{if $\exists (E_i,E_j) \in \mathcal{M}_F^2 \ \ E_i \neq E_j \ \ (x,y) \in E_i \times E_j$}.\\
  \end{array}
 \right.
$$

We see at once that $d$ is symmetrical. Since $d(x,y) \simeq
+\infty$ for every $(x,y) \in E_i \times E_j$ such that $E_i \neq
E_j$, we have
$$\forall (x,y) \in F^2 \ (d(x,y)=0 \Longleftrightarrow x=y)$$
It remains to prove the triangular inequality $d(x,z) \leq
d(x,y)+d(y,z)$ for each $(x,y,z)$ dans $F^3$. If $x,y,z$ belong to
the same metric component, there is no problem. Thus, we have to
consider four cases.

\vskip 2mm

\emph{Case 1}: $x \in E_i$ and $y,z \in E_j$ with $E_i \neq E_j$ in
$\mathcal{M}_F$.

Then, $d(x,z)=\delta(x,x_i)+d_{ij}+\delta(x_j,z)$,
$d(x,y)=\delta(x,x_i)+d_{ij}+\delta(x_j,y)$ and
$d(y,z)=\delta(y,z)$. Hence, the result come from $\delta(x_j,z)
\leq \delta (x_j),y) +\delta(y,z)$.

\vskip 2mm

\emph{Case 2}: $x,y \in E_i$ and $z \in E_j$ with $E_i \neq E_j$ in
$\mathcal{M}_F$.

Then, $d(x,z)=\delta(x,x_i)+d_{ij}+\delta(x_j,z)$,
$d(x,y)=\delta(x,y)$ and
$d(y,z)=\delta(y,x_i)+d_{ij}+\delta(x_j,z)$. Hence, the result comes
from $\delta(x,x_i) \leq \delta(x,y)+\delta(y,x_i)$.

\vskip 2mm

\emph{Case 3}: $x,z \in E_i$ and $y \in E_j$ with $E_i \neq E_j$ in
$\mathcal{M}_F$.

Then, $d(x,z)=\delta(x,z)$,
$d(x,y)=\delta(x,x_i)+d_{ij}+\delta(x_j,y)$ and
$d(y,z)=\delta(y,x_j)+d_{ji}+\delta(x_i,z)$. Since $\delta(x,z) \in
\Reel_+$, $d(x,y) \simeq +\infty$ and $d(y,z) \simeq +\infty$, we
get the result.

\vskip 2mm

\emph{Case 4}: $x \in E_i$, $y \in E_j$ and $z \in E_k$ with $E_i
\neq E_j$, $E_i \neq E_k$ and $E_j \neq E_k$ in $\mathcal{M}_F$.

Then, $d(x,z)=\delta(x,x_i)+d_{ik}+\delta(x_k,z)$,
$d(x,y)=\delta(x,x_i)+d_{ij}+\delta(x_j,y)$ and
$d(y,z)=\delta(y,x_j)+d_{jk}+\delta(x_k,z)$. Since $\Delta(E_i,E_k)
\leq \Delta(E_i,E_j)+\Delta(E_j,E_k)$ in $\Gal(\KK)=\KK/\Gal(0)$, we
have $r_{ik} \leq r_{ij}+r_{j,k}$ for every $r_{ik} \in
\Delta(E_i,E_k)$, $r_{ij} \in \Delta(E_i,E_j)$ and $r_{jk} \in
\Delta(E_j,E_k)$. Hence, the result comes from $d(x,z) \in
\Delta(E_i,E_k)$, $d(x,y) \in \Delta(E_i,E_j)$ and $d(y,z) \in
\Delta(E_j,E_k)$.

\vskip 2mm
 
Now, we know that $(F,d)$ is a $\KK$-metric space. It is easy to
check that the galactic projection of $(F,d)$ is isomorphic to
$(F,\delta,\Delta)$. 

\end{proof}

From the preceding result, we could hastily conclude that the study of galactic spaces may be advantageously replaced by the study of $\KK$-metric spaces. On the contrary, we think that galactic spaces are interesting because they have a rich metric structure which is a suitable framework for scaling. In the next section, we will introduce a notion of contraction which naturally operates on the class of galactic spaces.

\vskip 2mm

\textbf{Example~4} Let us call \emph{galactic continuous line}  the galactic space $D_{c,\KK}$ which is the galactic projection of $\KK$ view as a $\KK$-metric space. Thus, the set $D_{c,\KK}$ is equal to $\Hal(\KK)=\KK/\Hal(0)$, the generalized distance is given by  $$\forall (x,y) \in \KK \  \  \delta(\Hal(x),\Hal(y))=\text{vp}(|x-y|),$$
for each $x \in \KK$ the metric component of $\Hal(x)$ is $G(x):=\{\Hal(y) \ ; \ y \in \Gal(x)\}$, the set of metric components is the quotient group of $\Hal(\KK)$ by the subgroup $\Gal(0)/\Hal(0)$ (quotient which is canonically isomorphic to $\Gal(\KK)$)  and the galactic distance is given by
$$\forall (x,y) \in \KK^2 \ \ \Delta(G(x),G(y))=\Gal(|x-y|).$$
Since the principal value map is bijective if considered as a map from $\Gal(0)/\Hal(0)$ to $\Reel$ and since $G(x)=\Hal(x)+G(0)$, we see that each metric component of $D_{c,\KK}$ is a metric space isometric to $\Reel$ and $D_{c,\KK}$ is the disjoint union of a family $(\Reel_C)_{C \in \Gal(\KK)}$ of copies of $\Reel$.

In the same way, we define a \emph{galactic discrete line} $D_{d,\KK}$ to be the disjoint union of a family $(\Relatif_C)_{C \in \Gal(\KK)}$ where each $\Relatif_C$ is a copy of the set $\Relatif$ of integers. On $D_{d,\KK}$, we consider a generalized distance $\delta$ such that
$$\forall x,y  \in D_{d,\KK} \ \ \delta(x,y)=
\left\{
\begin{array}{ll}
 |x-y| & \text{if $\exists C \in \Gal(\KK)$ such that $(x,y) \in C^2$}  \\
  + \infty & \text{else} \\
\end{array}
\right.
$$
The metric components of $D_{d,\KK}$ for $\delta$ are the sets $\Relatif_C$ for $C \in \Gal(\KK)$ and the galactic distance is simply defined by 
$$\forall C,C'  \in \Gal(\KK) \ \ \Delta(\Relatif_C,\Relatif_{C'})=|C-C'|.$$
The discrete galactic line is the galactic space $(D_{d,\KK},\delta,\Delta)$.

\section{Contraction of a galactic space}

\textbf{4.1}  Before going further, we need some considerations on the
multiplication of a galaxy by a number. Let $\gamma \in \KK_+^*$
be a limited number, $\tau \in \Gal(\KK)$ and $t \in \KK$ such that $\tau=\text{Gal}(t)$. As usual, we define the set $\gamma.\tau=\{\gamma s \ ; \ s \in \tau\}$ so that
$\gamma.\tau=\gamma t+\gamma \Gal(0)$.

Thus, $\gamma.\tau=\text{Gal}(\gamma t)$ if $\gamma$ is appreciable
and $\gamma.\tau \subset \text{Hal}(\gamma t) \subset
\text{Gal}(\gamma t)$ if $\gamma \simeq 0$. Then, we define $\gamma \bullet \tau$ in $\Gal(\KK)$ by $\gamma
\bullet \tau:=\text{Gal}(\gamma t)$; hence, we always have $\gamma.\tau \subset \gamma
\bullet \tau$.

Moreover, if $\tau \in \Gal(\KK)_+^*$, the principal value map
is constant on the set $\gamma.\tau$ and this constant value is
named $\text{pv}(\gamma.\tau)$. Indeed, such a $\tau$ can be written
$\text{Gal}(t)$ for some $t \in \KK_+$ such that $t \simeq
+\infty$; therefore  every element of $\gamma.\tau$ is infinitely
large if $\gamma \not\simeq 0$ and $\gamma.\tau \subset
\text{Hal}(\gamma t)$ if $\gamma \simeq 0$.

\vskip 3mm

\textbf{4.2}  Let $(F,\delta,\Delta)$ and $(F',\delta',\Delta')$ be two galactic
spaces and let $f$ be a map $F \rightarrow F'$. We suppose that $f$
satisfies a Lipschitz condition for $(\delta,\delta')$, that is to
say, there is a limited element  $\gamma >0$ of $\KK$ such that
$$\forall (x_1,x_2) \in F^2 \ \
\delta'(f(x_1),f(x_2)) \leq \gamma \,\delta(x_1,x_2)$$ 
When one of its members is $+\infty$, this inequality
must be interpreted according to the following usual rule:  $\forall \alpha \in \KK \ \ \alpha \leq +\infty \ \ \text{and} \ \ \alpha (+\infty)=+\infty $, . If $f$ satisfies such a condition, then
$$\forall (x_1,x_2) \in F^2\ \ (\delta(x_1,x_2)<+\infty \ \Longrightarrow \
\delta'(f(x_1),f(x_2))<+\infty)$$ 
Thus, for every metric component $E$ of
$F$ for $\delta$, the direct image $f(E)$ is a subset of a metric component of
$F'$ for $\delta'$ which we denote $f[E]$. Hence, we get a map between the sets of metric components
$$\begin{array}{ccccc}
    \widetilde{f} & : & \mathcal{M}_F & \longrightarrow & \mathcal{M}_{F'} \\
      &   & E & \longmapsto & f[E] \\
  \end{array}
$$

\begin{definition}
Given two galactic spaces $(F,\delta,\Delta)$ and
$(F',\delta',\Delta')$, a \emph{morphism} from $(F,\delta,\Delta)$
to $(F',\delta',\Delta')$ is a map $f:F \rightarrow F'$ such that
there is a limited element $\gamma>0$ in $\KK$ so that the following Lipschitz conditions are satisfied
\begin{enumerate}
\item $\forall (x_1,x_2) \in F^2 \ \
\delta'(f(x_1),f(x_2)) \leq \gamma\,\delta(x_1,x_2)$;
\item $\forall (E_1,E_2) \in (\mathcal{M}_F)^2 \ \
\Delta'(\widetilde{f}(E_1),\widetilde{f}(E_2)) \leq \gamma \bullet
\Delta(E_1,E_2)$;
\end{enumerate}
The element $\gamma$ is called the Lipschitz constant of $f$.
An \emph{isomorphism} from $(F,\delta,\Delta)$ to
$(F',\delta',\Delta')$ is a morphism $f:(F,\delta,\Delta)
\rightarrow (F',\delta',\Delta')$ such that $f:F \rightarrow F'$ is
bijective and $f^{-1}$ is a morphism from $(F',\delta',\Delta')$ to
$(F,\delta,\Delta)$.
\end{definition}

We remark that if $f:(F,\delta,\Delta) \rightarrow
(F',\delta',\Delta')$ is an isomorphism, then $f[E]=f(E)$ for each
$E \in \mathcal{M}_F$. An isometry is a particular case of
isomorphism between two galactic spaces.

Hence, we have defined a category $\mathcal{G}_{\KK}$ whose
objects are the galactic spaces provided with the morphisms defined
just above. There is a class of morphisms which is particularly interesting for
the study of scalings.

\begin{definition}
Given a limited element $\gamma>0$ in $\KK$,  a
\emph{$\gamma$-contraction} of $(F,\delta,\Delta)$ is a morphism of
galactic spaces $f:(F,\delta,\Delta) \rightarrow
(F',\delta',\Delta')$ such that $f$ is a surjective map from $F$ to $F'$ and, $\forall (x_1,x_2) \in F^2$ $\forall
(E_1,E_2) \in \mathcal{M}_F^2$
\begin{align*}
\delta'(f(x_1),f(x_2)) & =
 \left\{\begin{array}{ll}
     \text{\rm pv}(\gamma\,\delta(x_1,x_2)) & \text{if $\delta(x_1,x_2)<+\infty$} \\
     \text{\rm pv}(\gamma.\Delta(C_F(x_1),C_F(x_2))) & \text{if $\delta(x_1,x_2)=+\infty$} \\
   \end{array}
 \right. \\
\Delta'(f[E_1],f[E_2]) & = \gamma \bullet \Delta(E_1,E_2)
\end{align*}
The element $\gamma$ is called the \emph{coefficient} of the contraction $f$.
\end{definition}

We also say that a galactic space $(F',\delta',\Delta')$ is a
$\gamma$-contraction of $(F,\delta,\Delta)$ if there exists a
morphism $f:(F,\delta,\Delta) \rightarrow (F',\delta',\Delta')$
which is a $\gamma$-contraction. We notice that a 1-contraction is an  isometry.
 It is easy to check that, if a surjective map $f:F \rightarrow F'$ satisfies the two last conditions of the preceding definition, then $f$ is necessary a morphism from $(F,\delta,\Delta)$ to $(F',\delta',\Delta')$. We point out that the coefficient $\gamma$ of a contraction may be greater that 1, but not to much.
 
\vskip 3mm

\textbf{4.3} Let us consider the first properties of contractions.
 
\begin{proposition} 
Let us consider a limited element $\gamma>0$ in $\KK$ and a $\gamma$-contraction $f_\gamma:(F,\delta,\Delta) \rightarrow (F_\gamma,\delta_\gamma,\Delta_\gamma)$.
\begin{enumerate}
\item If $\gamma \not\simeq 0$, then the map $f_\gamma:F \rightarrow F_\gamma$ is a bijection and, for every $x,y \in F$
\begin{align*}
\delta_\gamma(f_\gamma(x),f_\gamma(y)) & = \text{pv}(\gamma)\,\delta(x,y) \\
\Delta_\gamma(C_{F_\gamma}(f_\gamma(x)),C_{F_\gamma}(f_\gamma(y))) & =\gamma\bullet\Delta(C_F(x),C_F(y))
\end{align*}
\item If $\gamma \simeq 0$, then the map $f_\gamma$ is a surjective map such that, for every $u \in F$, we have 
$f_\gamma^{-1}(f_\gamma(u))=\{v \in F \ ; \  \gamma.\Delta(C_F(u),C_F(v)) \subset
\text{\rm Hal}(0)\}$; furthermore, for every $x,y \in F$
\begin{align*}
\delta_\gamma(f_\gamma(x),f_\gamma(y)) & =\text{pv}(\gamma.\Delta(C_F(x),C_F(y))) \\
\Delta_\gamma(C_{F_\gamma}(f_\gamma(x)),C_{F_\gamma}(f_\gamma(y))) & =\gamma\bullet\Delta(C_F(x),C_F(y))
\end{align*}
\end{enumerate}
\end{proposition}

\begin{proof}
Straightforward.
\end{proof}

\begin{theoreme}
Let $(F,\delta,\Delta)$ be a galactic space and a limited element $\gamma>0$ in $\KK$. Then, there is a
$\gamma$-contraction $f_\gamma:(F,\delta,\Delta) \rightarrow
(F_\gamma,\delta_\gamma,\Delta_\gamma)$ of $(F,\delta,\Delta)$.
\end{theoreme}

\begin{proof}
Let us consider the equivalence relation $\sim_\gamma$ on $F$ such
that, for all $(x,y) \in F^2$
$$x \sim_\gamma y \Longleftrightarrow \left\{
\begin{array}{ll}
  \gamma\,\delta(x,y) \simeq 0 & \text{if $\delta(x,y) < +\infty$} \\
  \gamma.\Delta(C_F(x)),C_F(y))\subset \text{Hal}(0) & \text{if $\delta(x,y)= +\infty$} \\
\end{array}\right.
$$
Let $F_\gamma$ be the quotient set $F/\sim_\gamma$ and $f_\gamma:F
\rightarrow F_\gamma$ be the canonical projection. Then, we define a
map $\delta_\gamma:F_\gamma \times F_\gamma \rightarrow \Reel$ so
that, for all $(x,y) \in F^2$
$$
\delta_\gamma(f_\gamma(x),f_\gamma(y))=
 \left\{
\begin{array}{ll}
  \text{pv}(\gamma\,\delta(x,y)) & \text{if $\delta(x,y) < +\infty$} \\
  \text{pv}(\gamma.\Delta(C_F(x),C_F(y))) & \text{if $\delta(x,y) =+\infty$} \\
\end{array}\right.
$$
which is clearly a generalized distance on $F_\gamma$. In a similar way, we define a galactic distance
$\Delta_\gamma$ on the set $\mathcal{M}_{F_\gamma}$ of metric components of $F_\gamma$ for $\delta_\gamma$ such that
$$\forall (x,y) \in F^2 \ \
\Delta_\gamma(C_{F_\gamma}(f_\gamma(x)),C_{F_\gamma}(f_\gamma(y)))=\gamma
\bullet \Delta(C_F(x),C_F(y))$$ Then $f_\gamma:(F,\delta,\Delta)
\rightarrow (F_\gamma,\delta_\gamma,\Delta_\gamma)$ is clearly a
$\gamma$-contraction of $(F,\delta,\Delta)$. 
\end{proof}

This proof shows that, given a galactic space $F$ and $\gamma$, the construction of a $\gamma$-contraction of $F$ is obtained by a relatively explicit procedure of quotient. 

\vskip 3mm

\textbf{4.4}
\emph{Example: contraction of the continuous and the discret galactic lines.} Given the continuous galactic line $D_{c,\KK}=(\Hal(\KK),\delta,\Delta)$ and an infinitesimal $\gamma$ such that $0<\gamma$, we want to construct a $\gamma$-contraction of $D_{c,\KK}$.  To this end, we consider the additive subgroup $\gamma^{-1}.\Hal(0):=\{\gamma^{-1}x \ ; \ x \in \Hal(0)\}$ of $\KK$, the quotient group $\gamma^{-1}\text{-}\Hal(\KK):=\KK/\gamma^{-1}.\Hal(0)$ whose elements are the sets $\gamma^{-1}\text{-}\Hal(x):=x+\gamma^{-1}.\Hal(0)$ for $x \in \KK$. On $\gamma^{-1}\text{-}\Hal(\KK)$, we define the generalized distance $\delta'$ by
$$\forall x,y \in \KK \ \ \delta'(\gamma^{-1}\text{-}\Hal(x),\gamma^{-1}\text{-}\Hal(y))=\text{pv}(\gamma |x-y|).$$
Each metric component is of the form $\gamma^{-1}\text{-}\Gal(x):=x+\gamma^{-1}.\Gal(0)$ for $x \in \KK$ and the set of metric components is the quotient group $\gamma^{-1}\text{-}\Gal(\KK):=\KK/\gamma^{-1}.\Gal(0)$. Then, we consider the galactic distance $\Delta'$ defined on $\gamma^{-1}\text{-}\Gal(\KK)$ by
$$\forall x,y \in \KK \ \ \Delta'(\gamma^{-1}\text{-}\Gal(x),\gamma^{-1}\text{-}\Gal(y))=\Gal(\gamma |x-y|).$$
Hence,  the map 
$$
\begin{array}{ccccc}
 f' & : & \Hal(\KK) & \longrightarrow & \gamma^{-1}\text{-}\Hal(\KK) \\
    &   &  \Hal(x)   &  \longmapsto     & \gamma^{-1}\text{-}\Hal(x)
 \end{array}
$$
is clearly a $\gamma$-contraction of $D_{c,\KK}$. Since $f'$ is not injective (for instance $f'(0)=f'(1)$), this map is not an isometry. Nevertheless, the map 
$$
\begin{array}{ccc}
 \Hal(\KK) & \longrightarrow & \gamma^{-1}\text{-}\Hal(\KK) \\
 \Hal(x)   &  \longmapsto     & \gamma^{-1}\text{-}\Hal(\gamma x)
\end{array}
$$
is an isometry of the galactic space $D_{c,\KK}$ to its $\gamma$-contraction. Thus, \emph{a $\gamma$-contraction of a continuous galactic line is isometric to itself}.

Let us now consider the case of the discrete galactic line $D_{d,\KK}$ . To this end, for each $C \in \Gal(\KK)$ we arbitrarily choose an element $x_C \in C$ and we define the map
$f'': D_{d,\KK} \rightarrow \gamma^{-1}\text{-}\Hal(\KK)$ such that, for each $x \in D_{d,\KK}$, we have $f''(x)= \gamma^{-1}\text{-}\Hal(x_C)$ where $x \in \Relatif_C$. Then, it is clear that $f''$ is a $\gamma$-contraction from $(D_{d,\KK},\delta,\Delta)$ to $(\gamma^{-1}\text{-}\Hal(\KK),\delta',\Delta')$. Since this last galactic space is isometric to $D_{c,\KK}$, we see that \emph{for $\gamma \simeq 0$, a $\gamma$-scaling of the discrete galactic line is isometric to the continuous galactic line.}

\vskip 3mm

\textbf{4.5} Now, we want to understand the relations between the different contractions of a given galactic space.
\begin{proposition}
We consider two limited elements $\alpha, \beta  \in \KK$ such that $0 < \beta \leq \alpha$ and two morphisms of galactic spaces  
$f_\alpha:(F,\delta,\Delta) \rightarrow (F_\alpha,\delta_\alpha,\Delta_\alpha)$  and $f_\beta:(F,\delta,\Delta) \rightarrow (F_\beta,\delta_\beta,\Delta_\beta)$ defined on the same space such that $f_\alpha$ is an $\alpha$-contraction and $f_\beta$ is a $\beta$-contraction. Then, there is an unique map $f_{\beta,\alpha}:F_\alpha \rightarrow F_\beta$ such that $f_\beta=f_{\beta,\alpha} \circ f_\alpha$. Moreover, $f_{\beta,\alpha}$ is a $\beta/\alpha$-contraction $(F_\alpha,\delta_\alpha,\Delta_\alpha) \rightarrow (F_\beta,\delta_\beta,\Delta_\beta)$.
\end{proposition}

We say that $f_{\beta,\alpha}$ is the \emph{transition} between the two contractions $(F_\alpha,\delta_\alpha,\Delta_\alpha)$ and  $(F_\beta,\delta_\beta,\Delta_\beta)$ of $(F,\delta,\Delta)$.

\begin{proof}
For all $(x_1,x_2) \in F$, if $\delta_\alpha(f_\alpha(x_1),f_\alpha(x_2))=0$ then $\delta_\beta(f_\beta(x_1),f_\beta(x_2))=0$; since $f_\alpha$ is surjective, we deduce that there exists an unique map $f_{\beta,\alpha}:F_\alpha \rightarrow F_\beta$ such that $f_\beta=f_{\beta,\alpha} \circ f_\alpha$. It is easy to check that $f_{\beta,\alpha}$ is a $\beta/\alpha$ contraction.
\end{proof}

\begin{corollaire}
For each limited element $\gamma>0$ in $\KK$, two $\gamma$-contractions of a same galactic space are isometric.
\end{corollaire}

If a galactic space $(F_0,\delta_0,\Delta_0)$ is such that the $F_0$ has only one element, then $\delta_0$  and $\Delta_0$ are trivial and we say that $(F_0,\delta_0,\Delta_0)$ is a trivial galactic space. We notice that, for each galactic space $(F,\delta,\Delta)$, there is an unique morphism from $(F,\delta,\Delta)$ to $(F_0,\delta_0,\Delta_0)$. The next result shows that the limit when $\gamma \rightarrow 0$ of the $\gamma$-contractions of a galactic space is a trivial galactic space.

\begin{proposition}
Let us consider a galactic space $(F,\delta,\Delta)$ and a family of galactic spaces $\{(F_\gamma,\delta_\gamma,\Delta_\gamma)\}_{0 < \gamma \leq 1}$ such that, for each $0 < \gamma \leq 1$ in $\KK$, $(F_\gamma,\delta_\gamma,\Delta_\gamma)$ is a $\gamma$-contraction of $(F,\delta,\Delta)$. For every $\alpha, \beta \in \KK_+^*$ such that $0 < \beta \leq \alpha \leq 1$, let $f_{\beta,\alpha}$ be the \emph{transition} between the two contractions $(F_\alpha,\delta_\alpha,\Delta_\alpha)$ and  $(F_\beta,\delta_\beta,\Delta_\beta)$ of $(F,\delta,\Delta)$. Then,  in the category $\mathcal{G}_{\KK}$, the family  $\{f_{\beta, \alpha}\}_{0<\beta<\alpha \leq 1}$ has a direct limit which is a trivial galactic space. 
\end{proposition}

\begin{proof}
From the preceding proposition, we know that the product of two transitions is a transition. Now, we choose a trivial galactic space $(F_0,\delta_0,\Delta_0)$ and for each  $0 < \gamma \leq 1$ in $\KK$, let $f_{0,\gamma}:(F_\gamma,\delta_\gamma,\Delta_\gamma) \rightarrow (F_0,\delta_0,\Delta_0)$ be a trivial morphism. Then, for any $\alpha, \beta  \in \KK$ such that $0 < \beta \leq \alpha \leq 1$, we have $f_{0,\alpha}=f_{0,\beta} \circ f_{\beta,\alpha}$. 

Now, we consider a galactic space $(G,d,D)$ and we suppose that, for each $\alpha  \in \KK$ such that $0 <  \alpha \leq 1$ and for each $\alpha$-contraction $(F_\alpha,\delta_\alpha,\Delta_\alpha)$  of $(F,\delta,\Delta)$, we have a morphism $g_\alpha:(F_\alpha,\delta_\alpha,\Delta_\alpha) \rightarrow (G,d,D)$ such that $g_\alpha=g_\beta \circ f_{\beta,\alpha}$ for every element $\beta$ of $\KK$ such that $0 < \beta \leq \alpha$ and for every $\beta$-contraction $(F_\beta,\delta_\beta,\Delta_\beta)$ of $(F,\delta,\Delta)$ with transition $f_{\beta,\alpha}$. Given two points $x_1$ and $x_2$ in $F_\alpha$ we can find a sufficiently small $\beta_0 < \alpha$ such that $\delta_\beta(f_{\beta,\alpha}(x_1),f_{\beta,\alpha}(x_2))=0$ for every $\beta \leq \beta_0$. Hence, we see that there is a single element $y_0 \in G$ such that $g_\alpha(F_\alpha)=\{y_0\}$ for each $\alpha$. Consequently, the constant map $g_0:F_0 \rightarrow G$ with value $y_0$ is the unique morphism such that $g_\alpha=g_0 \circ f_{0,\alpha}$ for each $\alpha$.
\end{proof}

\vskip 3mm

\textbf{4.6}
In the following result, for each galactic space $(F,\delta,\Delta)$, the set $F$ is
provided with the topology defined by the generalized distance $\delta$ and the set $\mathcal{M}_F$ of its metric components is provided with the topology defined by the galactic distance $\Delta$.

\begin{proposition} 
Let us consider a limited element $\gamma>0$ in $\KK$ and a $\gamma$-contraction $f_\gamma:(F,\delta,\Delta) \rightarrow (F_\gamma,\delta_\gamma,\Delta_\gamma)$.
\begin{enumerate}
\item The maps $f_\gamma:F \rightarrow F_\gamma$ and
$\widetilde{f}_\gamma:\mathcal{M}_F \rightarrow \mathcal{M}_{F_\gamma}$ are continuous.
\item If $\gamma \not\simeq 0$, $f_\gamma$ and $\widetilde{f}_\gamma$ 
are homeomorphisms.
\item If $\gamma \simeq 0$, then for every
$z \in F_\gamma$,  ${f_\gamma}^{-1}(\{z\})$ is an
open set of $F$.
\item If $\gamma \simeq 0$ and if we can find $\eta \in \KK_+^*$ such that $\eta \simeq 0$ and $\gamma/\eta \simeq 0$, then, for every
$Z \in \mathcal{M}_{F_\gamma}$  the set
${\widetilde{f}_\gamma}^{-1}(\{Z\})$ is an open set of
$\mathcal{M}_F$.
\end{enumerate}
\end{proposition}

\begin{proof}
1. Let $V$ be an open set of $F_\gamma$ and $z \in
(f_{\gamma})^{-1}(V)$. Let $z' \in V$ such that
$f_{\gamma}(z)=z'$ and $r \in \Reel_+^*$ such that the
open ball $B_{\delta_\gamma}(z',r)$ is a subset of $V$. Then,
from the condition $\gamma \leq 1$ we deduce that 
$B_{\delta}(z,r) \subset (f_{\gamma})^{-1}(B_{\delta_\gamma}(z',r))$.
Hence,  the set ${f_{\gamma}}^{-1}(V)$ is an open set of
$F_\gamma$. Thus $f_\gamma$ is continuous. A similar argument shows that $\widetilde{f}_\gamma$ is also continuous.

\vskip 2mm

2. We suppose that $\gamma \not\simeq 0$. Then, we know that $f_\gamma$ is invertible and, for every $x',y' \in F_\gamma$
\begin{align*}
\delta(f_\gamma^{-1}(x'),f_\gamma^{-1}(y')) & = \text{pv}(\gamma)^{-1}\,\delta_\gamma(x',y') \\
\Delta(C_{F}(f_\gamma^{-1}(x')),C_{F}(f_\gamma^{-1}(y'))) & =\gamma^{-1} \bullet\Delta_\gamma(C_{F_\gamma}(x'),C_{F_\gamma}(y'))
\end{align*}
From this, we deduce that $f_\gamma^{-1}$ and $(\widetilde{f}_\gamma)^{-1}$ are continuous.

3. We suppose now that $\gamma \simeq 0$. Given $z' \in
F_\gamma$, we consider any $z \in F$ such that
$f_{\gamma}(z)=z'$. Let $E \in \mathcal{M}_F$ such
that $z \in E$ and let $t$ an arbitrary point of $E$. Since $\delta(z,t)$ is limited we see that $\delta_\gamma(f_\gamma(z),f_\gamma(t))=0$. Thus $E \subset f_\gamma^{-1}(\{z'\})$ and since $E$ is a neighborhood of $z$ in $F$, we get that $f_\gamma^{-1}(\{z'\}))$ is open.

4. We suppose that $\gamma \simeq 0$ and that  we can find $\eta \in \KK_+^*$ such that $\eta \simeq 0$ and $\gamma/\eta \simeq 0$ . We consider
$Z \in \mathcal{M}_{F_\gamma}$ and let $E \in \mathcal{M}_F$ such that
$\widetilde{f}_{\gamma}(E)=Z$. Then, $\rho=\Gal(\eta^{-1})$ is strictly greater than $0:=\Gal(0)$ in the ordered group $\Gal(\KK)$. We consider an element $H$ of the 'open ball' $B_{\Delta}(E,\rho)$ of $\mathcal{M}_F$. Thus, 
$\Delta(E,H)<\rho$ ; consequently, $H \in {\widetilde{f}_\gamma}^{-1}(\{Z\}$ since $\Delta_\gamma(\widetilde{f}_\gamma(E),\widetilde{f}_\gamma(H))=0$ because $\gamma \bullet \Delta(E,H) \leq \gamma \bullet \rho = \Gal(\gamma/\eta) = \Gal(0)=0$ . Hence, $B_\Delta(E,\rho) \subset {\widetilde{f}_\gamma}^{-1}(\{Z\})$ and this last set is open.
\end{proof}
If $\gamma \simeq 0$, we notice that the property
$$\text{$\exists \eta \in \KK_+^*$ tel que $\eta \simeq 0$ et $\gamma/\eta \simeq 0$}$$
 is not  satisfied in every ordered field extension of $\Reel$: for instance, if we choose $\gamma:=X$, we cannot find such a $\eta$ in the field or rational functions $\Reel(X)$ or in the field or Laurent series $\Reel((X))$. In the field $\Reel[[X^\Rationnel]]$ of Puiseux series or in a field ${^*\Reel}$ of hyperreal numbers, this property is  true for every $\gamma \simeq 0$.

\section{Nonstandard scaling of a metric space}

In all this section, we consider a metric space $(X,d)$. We need some nonstandard extensions of $\Reel$, $X$ and $d$. To this end, we can use the method of ultra-powers as in section~2. More generally \cite{Hens87,ANSLoeb}, we can consider a
superstructure $V(S)$ over a set $S$ such that $(X \cup \Reel)
\subset S$ and a nonstandard model of $V(S)$ 
$$\begin{array}{ccc}
    V(S) & \longrightarrow & V({^*S}) \\
    Y & \longmapsto & {^*Y} \\
  \end{array}$$
with a large enough saturation property (our study does not require any particular refinement in the choice of the nonstandard model). Equivalently, we can used the axiomatic approach of Hrba\v{c}ek \cite{Hrba79}. Notice that  Nelson's internal set theory IST \cite{Nelson} is not adapted to this work since it does not allow a convenient treatment of external sets.

Then we get at the same time the nonstandard extensions ${^*\Reel}$ or $\Reel$, 
${^*X}$ of $X$ and
${^*\!d}:{^*X} \times {^*X} \rightarrow {^*\Reel}_+$ of $d$  provided by the given
nonstandard model.

\vskip 3mm

\textbf{5.1}  Each element of the multiplicative group ${^*\Reel}_+^*:=\{\gamma
\in {^*\Reel} \ ; \ 0<\gamma \}$ of strictly positive hyperreal
numbers is called a \emph{scale}.

Given a scale $\alpha \in {^*\Reel}_+^*$, we define the equivalence
relation $\simeq_\alpha$ on ${^*X}$ defined by
$$\forall (x,y) \in {^*X}^2 \ \left( x \simeq_\alpha y \
\Longleftrightarrow \alpha\,{^*\!d}(x,y) \simeq 0\right)$$ Then, we
introduce the quotient set $X_\alpha={^*X}/\simeq_\alpha$ and the
canonical projection
$$\begin{array}{ccccc}
    \pi_\alpha & : & {^*X} & \longrightarrow & X_\alpha \\
                   &   &    x      & \longmapsto  & \pi_\alpha(x) \\
  \end{array}$$
where $\pi_\alpha(x)$ denotes the equivalence class of $x \in
{^*X}$, i.e the set of $y \in {^*X}$ such that $x \simeq_\alpha
y$.

Associated to the hyper-distance ${^*\!d}$, there is a natural map
$$
\begin{array}{ccccc}
  \delta_\alpha & : & X_\alpha \times X_\alpha & \longrightarrow & \Reel_+ \cup \{+\infty\} \\
              &   & (\xi,\eta) & \longmapsto & \delta_\alpha(\xi,\eta) \\
\end{array}
$$
such that, if $\xi=\pi_\alpha(x)$ and $\eta=\pi_\alpha(y)$, then
$\delta_\alpha(\xi,\eta)=\text{st}(\alpha\,{^*\!d}(x,y))$.

It is clear that $\delta_\alpha$ is a generalized distance on
$X_\alpha$. Let $\mathcal{M}_\alpha$ be the set of metric components
of $X_\alpha$ for $\delta_\alpha$. Thus each $E \in
\mathcal{M}_\alpha$ is of the following form
$$E=\text{Cone}(X,x_E,\alpha):=\left\{x \in {^*X} \ ; \ \alpha\,{^*\!d}(x,x_E) \not\simeq
+\infty\right\}/\simeq_\alpha$$ where $x_E$ is any point in ${^*X}$
such that $\pi_\alpha(x_E) \in E$. When $\alpha \simeq 0$, the set
$\text{Cone}(X,x_E,\alpha)$ is exactly the so-called
\emph{asymptotic cone} of $(X,d)$ with respect to $x_E$ and
$\alpha$.

We recall that $\text{Gal}$ is the canonical projection ${^*\Reel}
\rightarrow \Gal({^*\Reel})={^*\Reel}/\Gal(0)$. For every $E_1$ and
$E_2$ in $\mathcal{M}_\alpha$, we choose $x_{E_1}$ and $x_{E_2}$ in
${^*X}$ such that $E_1=\text{Cone}(X,x_{E_1},\alpha)$ and
$E_2=\text{Cone}(X,x_{E_2},\alpha)$; then we define
$$\Delta_\alpha(E_1,E_2)=\text{Gal}(\alpha\,{^*\!d}(x_{E_1},x_{E_2}))$$
Thus, we get a map $\Delta_\alpha:\mathcal{M}_\alpha^2 \rightarrow
\Gal({^*\Reel})$ which is a galactic distance and we can consider the galactic space
$(X_\alpha,\delta_\alpha,\Delta_\alpha)$.

\begin{definition}
Given a scale $\alpha \in {^*\Reel}_+^*$, the (nonstandard) \emph{$\alpha$-scaling} of the metric space $(X,d)$ is the
galactic space $(X_\alpha,\delta_\alpha,\Delta_\alpha)$.
\end{definition}

Hence, starting from a usual metric space $(X,d)$, we get a family $(X_\alpha,\delta_\alpha,\Delta_\alpha)_{\alpha \in {^*\Reel}_+^*}$ of galactic spaces which are the different scaling of $(X,d)$.

\vskip 3mm

\textbf{5.2}  Let us now consider two scales $\alpha,\beta \in {^*\Reel}_+^*$ such
that $\beta < \alpha$. Now, we want to compare the
$\alpha$-scaling and the $\beta$-scaling of our metric space
$(X,d)$.

If $(x,y) \in {^*X}^2$ is such that $\alpha\,{^*\!d}(x,y) \simeq 0$,
then $\beta\,{^*\!d}(x,y) \simeq 0$. Therefore, there exists a
natural surjective map $\pi_{\beta,\alpha} : X_\alpha \rightarrow
X_\beta$ such that $\pi_\beta=\pi_{\beta,\alpha} \circ \pi_\alpha$.
In the same way, if $\gamma \in {^*\Reel}_+^*$ is such that $\gamma
\leq \beta \leq \alpha$, then
$\pi_{\gamma,\alpha}=\pi_{\gamma,\beta} \circ \pi_{\beta,\alpha}$.

\begin{theoreme}
The map $\pi_{\beta,\alpha}$ is a  $(\beta/\alpha)$-contraction from the $\alpha$-scaling $(X_\alpha,\delta_\alpha,\Delta_\alpha)$ of $(X,d)$ onto its $\beta$-scaling $(X_\beta,\delta_\beta,\Delta_\beta)$.
\end{theoreme}

In other words, the $\beta$-scaling of $(X,d)$ is a
$(\beta/\alpha)$-contraction of its $\alpha$-scaling. Consequently,
insofar as we are only concerned by the structure of galactic space,
we can define the $\beta$-scaling of $(X,d)$ using only its
$\alpha$-scaling.

\begin{proof}

It is clear that $\pi_{\beta,\alpha}$ is a surjective map from $X_\alpha$ to $X_\beta$.
Furthermore, if $\xi=\pi_\alpha(x)$ and $\eta=\pi_\alpha(y)$, then
$$\delta_\alpha(\xi,\eta)=\text{st}(\alpha\,{^*\!d}(x,y))\
\text{and} \
\delta_\beta(\pi_{\beta,\alpha}(\xi),\pi_{\beta,\alpha}(\eta))=
\text{st}(\beta\,{^*\!d}(x,y))$$ Therefore
$$\delta_\beta(\pi_{\beta,\alpha}(\xi),\pi_{\beta,\alpha}(\eta))=
\text{st}((\beta/\alpha)\,\alpha{^*\!d}(x,y))$$ hence $\displaystyle
\delta_\beta(\pi_{\beta,\alpha}(\xi),\pi_{\beta,\alpha}(\eta)) =
 \left\{\begin{array}{ll}
     \text{st}((\beta/\alpha)\,\delta_\alpha(\xi,\eta)) &
             \text{if $\delta_\alpha(\xi,\eta)<+\infty$} \\
     \text{st}(\gamma.\Delta(C_{X_\alpha}(\xi),C_{X_\alpha}(\eta)))
     & \text{if $\delta_\alpha(\xi,\eta)=+\infty$} \\
   \end{array}
 \right.$.

In the same way, if $E=\text{Cone}(X,x,\alpha)$ and
$F=\text{Cone}(X,y,\alpha)$, then
$$\Delta_\beta(\pi_{\beta,\alpha}[E],\pi_{\beta,\alpha}[F])
=\text{Gal}(\beta\,{^*\!d}(x,y))=
(\beta/\alpha)\bullet\Delta_\alpha(E,F)$$

\end{proof}

From this, it results that the maps $\pi_{\beta,\alpha}$ have all the properties of transitions stated in the preceding section. For instance, the limit of $(X_\alpha,\delta_\alpha,\Delta_\alpha)$ when $\alpha \rightarrow 0$ in $\KK_+^*$ is a trivial galactic space.

\vskip 3mm

\textbf{5.3}  A new feature about the family of nonstandard scalings of a metric space $(X,d)$ is that the galactic spaces $(X_\alpha,\delta_\alpha,\Delta_\alpha)$ are defined for arbitrary large scale $\alpha \in {^*\Reel}_+^*$. Then, a natural question is related to the existence of the limit of $(X_\alpha,\delta_\alpha,\Delta_\alpha)$ when $\alpha \rightarrow +\infty$ in ${^*\Reel}_+^*$.

Let us call a \emph{chain} any family
$\xi=(\xi_\alpha)_{\alpha \in {^*\!\Reel}_+^*}$ such that, for each
$\alpha \in {^*\!\Reel}_+^*$ the element $\xi_\alpha$ belongs to
$X_\alpha$ and $\pi_{\beta,\alpha}(\xi_\alpha)=\xi_\beta$ for each
$\beta \leq \alpha \in {^*\!\Reel}_+^*$.

\begin{proposition}
Let us consider two chains $\xi=(\xi_\alpha)_{\alpha \in
{^*\!\Reel}_+^*}$ and $\xi'=(\xi_\alpha')_{\alpha \in
{^*\!\Reel}_+^*}$ such that $\xi \neq \xi'$. For every $\alpha \in {^*\!\Reel}_+^*$, let
$E_\alpha$  and $E_\alpha'$ be the metric components of respectively  $\xi_\alpha$ and $\xi_\alpha'$ in $X_\alpha$. Then, there is
$\alpha_0 \in {^*\!\Reel}_+^*$ such that, for every $\alpha \geq
\alpha_0$, we have $\delta_\alpha(\xi_\alpha,\xi_\alpha')=+\infty$.
Furthermore, $\displaystyle \lim_{\alpha \rightarrow +\infty} \Delta_{\alpha}(E_\alpha,E_\alpha')=+\infty$
for the order topology on ${^*\!\Reel}_+^*$ and $\Gal({^*\!\Reel})$.
\end{proposition}

\begin{proof}
For each $\alpha \in {^*\!\Reel}_+^*$, we choose $x_\alpha,x_\alpha'
\in {^*\!X}$ such that $\pi_\alpha(x_\alpha)=\xi_\alpha$ and
$\pi_\alpha(x_\alpha')=\xi_\alpha'$. Since $\xi \neq \xi'$, there is
$\beta \in {^*\!\Reel}_+^*$ such that $\xi_\beta \neq \xi_\beta'$.
Thus $\beta\,{^*\!d}(x_\beta,x_\beta') \not\simeq 0$. Let $\alpha$ a scale such that $\alpha > \beta$; since
$\pi_\beta(x_\alpha)=\pi_{\beta,\alpha} \circ
\pi_\alpha(x_\alpha)=\pi_\beta(x_\beta)$ and
$\pi_\beta(x_\alpha')=\pi_{\beta,\alpha} \circ
\pi_\alpha(x_\alpha')=\pi_\beta(x_\beta')$, we have
$\beta\,{^*\!d}(x_\alpha,x_\beta) \simeq 0$ and
$\beta\,{^*\!d}(x_\alpha',x_\beta') \simeq 0$ and thus
$\beta\,{^*\!d}(x_\alpha,x_\alpha') \not\simeq 0$. Hence, if we
choose $\alpha_0$ such that $\alpha_0/\beta \simeq +\infty$, we have
for every $\alpha \geq \alpha_0$
$$\delta_\alpha(\xi_\alpha,\xi_\alpha')=
\text{st}(\alpha\,{^*\!d}(x_\alpha,x_\alpha'))=
\text{st}((\alpha/\beta)\beta\,{^*\!d}(x_\alpha,x_\alpha'))=+\infty$$
Moreover, since $\displaystyle \lim_{\alpha \rightarrow +\infty}
\Delta_\alpha(E_\alpha,E_\alpha')=\text{Gal}(\alpha\,{^*\!d}(x_\alpha,x_\alpha'))
=\text{Gal}((\alpha/\beta)\beta\,{^*\!d}(x_\alpha,x_\alpha'))$ we
see that $\Delta_\alpha(E_\alpha,E_\alpha')$ converges towards
$+\infty$ in $\Gal({^*\!\Reel})$ when $\alpha \rightarrow +\infty$
in ${^*\Reel}_+^*$.
\end{proof}

The last result suggests that, if we want to find a limit for 
$(X_\alpha,\delta_\alpha,\Delta_\alpha)$  when $\alpha
\rightarrow +\infty$, we have to widen the category
$\mathcal{G}_{{^*\!\Reel}}$ of galactic spaces. To this end, we introduce the category
$\mathcal{G}'_{{^*\!\Reel}}$ of \emph{generalized
galactic spaces}. The objects of this category are structures
$(X,\delta,\Delta)$ where $X$ is a set, $\delta$ is a generalized
distance on $X$ and $\Delta$ is a \emph{generalized galactic
distance} on the set $\mathcal{M}_X$ of metric components of $X$ for
$\delta$. This last condition means that $\Delta$ is a map
$\mathcal{M}_X \times \mathcal{M}_X \rightarrow \Gal({^*\!\Reel})_+^* \cup
\{+\infty\}$ which satisfies the general metric rule. In
$\mathcal{G}'_{{^*\!\Reel}}$, a morphism from
$(X,\delta,\Delta)$ to $(X',\delta',\Delta')$ is a map $f:X
\rightarrow X'$ such that there is a limited element $\gamma>0$ in $\KK$ so that the
following conditions are satisfied
\begin{enumerate}
\item $\forall (x_1,x_2) \in X^2 \ \
\delta'(f(x_1),f(x_2)) \leq \gamma\,\delta(x_1,x_2) $\\
(thus, $f$ induces a map $\widetilde{f}$ from the set
$\mathcal{M}_X$ of metric components of $X$ for $\delta$ to the set
$\mathcal{M}_{X'}$ of metric components of $X'$ for $\delta'$);
\item $\forall (E_1,E_2) \in {\mathcal{M}_X}^2 \ \
\Delta'(\widetilde{f}(E_1),\widetilde{f}(E_2)) \leq
\gamma \bullet\Delta(E_1,E_2)$.
\end{enumerate}

We denote by $\displaystyle
\lim_{\alpha \rightarrow +\infty} (X_\alpha,\delta_\alpha,\Delta_\alpha)$ the inverse (or projective) limit of the family $(\pi_{\beta,\alpha})_{\text{$\beta \leq
\alpha \in {^*\!\Reel}_+^*$}}$ of morphisms $\pi_{\beta,\alpha}:(X_\alpha,\delta_\alpha,\Delta_\alpha) \rightarrow (X_\beta,\delta_\beta,\Delta_\beta)$ in the category $\mathcal{G}'_{{^*\!\Reel}}$. If this limit exists, it is well defined up to an isomorphism.

\begin{proposition}
For each metric space $(X,d)$, the limit $\displaystyle \lim_{\alpha
\rightarrow +\infty} (X_\alpha,\delta_\alpha,\Delta_\alpha)$ exists
and is equal to the pair $\displaystyle
\left((X_\infty,\delta_\infty,\Delta_\infty),(\pi_{\alpha,\infty})_{\alpha
\in {^*\!\Reel}_+\*}\right)$ where $(X_\infty,\delta_\infty,\Delta_\infty)$ is a generalized galactic space and $(\pi_{\alpha,\infty})_{\alpha
\in {^*\!\Reel}_+\*}$ is a family of morphisms from $(X_\infty,\delta_\infty,\Delta_\infty) $ to  $(X_\alpha,\delta_\alpha,\Delta_\alpha)$ such that
\begin{itemize}
\item $X_\infty$ is the set of chains (families $\xi=(\xi_\alpha)_{\alpha \in
{^*\!\Reel}_+^*}$ such that $\xi_\alpha \in X_\alpha$ for all
$\alpha \in {^*\!\Reel}_+\*$ and
$\pi_{\beta,\alpha}(\xi_\alpha)=\xi_\beta$ for all $\beta \leq
\alpha \ \in {^*\!\Reel}_+\*$).
\item $\forall (\xi,\eta) \in X_{\infty}^2 \ \ \ \delta_\infty(\xi,\eta)=
 \left\{
  \begin{array}{cc}
    0 & \text{if $\xi=\eta$} \\
    +\infty & \text{if $\xi \neq \eta$} \\
  \end{array}
 \right.$
\item $\forall (\xi,\eta) \in X_{\infty}^2 \ \ \ \Delta_\infty(\{\xi\},\{\eta\})=
 \left\{
  \begin{array}{cc}
    0 & \text{if $\xi=\eta$} \\
    +\infty & \text{if $\xi \neq \eta$} \\
  \end{array}
 \right.$
\item $\forall \xi=(\xi_\alpha)_{\alpha \in {^*\!\Reel}_+^*} \in
X_\infty \ \forall \alpha \in {^*\!\Reel}_+^* \ \ \
\pi_{\alpha,\infty}(\xi)=\xi_\alpha$
\end{itemize}
\end{proposition}

This result says that the limit of the $\alpha$-scaling of $(X,d)$
when $\alpha$ approaches $+\infty$ is a huge generalized galactic
space in which the distance between two different points is
$+\infty$ and the galactic distance between two different metric
components is $+\infty$. If ${^*\!X}$ denotes the nonstandard
extension of $X$ used in the construction of the scaling of $(X,d)$,
we see that there is a natural map
$$
 \begin{array}{ccc}
   {^*\!X} & \longrightarrow & X_\infty \\
   x & \longmapsto & (\pi_{\alpha}(x))_{\alpha \in {^*\!\Reel}_+^*} \\
 \end{array}
$$
which is clearly injective but we do not know if it is surjective.

\begin{proof}[Proof of the proposition]
We remark that the set of metric components of $X_\infty$ is
$$\mathcal{M}_{X_\infty}=\{ \{\xi\} \ ; \ \xi \in X_\infty\}$$
It is clear that each $\pi_{\alpha,\infty}$ is a map $X_\infty
\rightarrow X_\alpha$ such that
$\pi_{\beta,\infty}=\pi_{\beta,\alpha} \circ \pi_{\alpha,\infty}$
whenever $\beta \leq \alpha \in {^*\!\Reel}_+\*$. In an obvious way,
$\pi_{\alpha,\infty}$ is a morphism of generalized galactic space
(with a Lipschitz constant equal to 1 for instance) from
$(X_\infty,\delta_\infty,\Delta_\infty)$ to
$(X_\alpha,\delta_\alpha,\Delta_\alpha)$.

Let $(Y,d,D)$ be a generalized galactic space and
$(\psi_\alpha)_{\alpha \in {^*\!\Reel}_+\*}$ be a family of
morphisms $\psi_\alpha$ from $(Y,d,D)$ to
$(X_\alpha,\delta_\alpha,\Delta_\alpha)$ such that $\psi_\beta=
\pi_{\beta,\alpha} \circ \psi_\beta$ for all $\beta \leq \alpha \in
{^*\!\Reel_+^*}$. For each $y \in Y$, the family
$\psi_\infty(y)=(\psi_\alpha(y))_{\alpha \in {^*\!\Reel}_+\*}$
belongs to $X_\infty$. Thus, we get a map $\psi_\infty:Y \rightarrow
X_\infty$ which is the unique map which satisfies
$\psi_\alpha=\pi_{\alpha,\infty} \circ \psi_\infty$ for all
 $\alpha \in {^*\!\Reel}_+\*$. Let $y$ and $y'$ two points of $Y$
 such that $\psi_\infty(y) \neq \psi_\infty(y')$. Hence, there is
 $\alpha_0 \in {^*\!\Reel}_+\*$ such that $\psi_\alpha(y) \neq
 \psi_\alpha(y')$ for every $\alpha \geq \alpha_0$. For sufficiently
large $\alpha$, we see that
$\delta_\alpha(\psi_\alpha(y),\psi_\alpha(y')=+\infty$ and thus
$d(y,y')=+\infty$. Consequently, $\psi_\infty$ induces a map
$\widetilde{\psi}_\infty : \mathcal{M}_Y \rightarrow
\mathcal{M}_{X_\infty}$. Let $E,E' \in \mathcal{M}_Y$ such that
$\widetilde{\psi}_\infty(E) \neq \widetilde{\psi}_\infty(E')$. There
are $\xi \neq \xi' \in X_\infty$ such that
$\widetilde{\psi}_\infty(E)=\{\xi\}$ and
$\widetilde{\psi}_\infty(E')=\{\xi'\}$. For each $\alpha \in
{^*\!\Reel}_+^*$, there is a real number $k_\alpha>0$ such that
$$\Delta_\alpha(\widetilde{\psi}_\alpha(E),\widetilde{\psi}_\alpha(E'))
\leq k_\alpha\,D(E,E')$$ and
$\Delta_\alpha(\widetilde{\psi}_\alpha(E),\widetilde{\psi}_\alpha(E'))=
\text{Gal}(\alpha\,{^*\!d}(x_\alpha,x_\alpha'))$ where
$x_\alpha,x_\alpha'$ are elements of ${^*\!X}$ such that
$\pi_\alpha(x_\alpha)=\pi_{\alpha,\infty}(\xi)$ and
$\pi_\alpha(x_\alpha')=\pi_{\alpha,\infty}(\xi')$. Since
$\text{Gal}(\alpha\,{^*\!d}(x_\alpha,x_\alpha'))$ is arbitrarily
large in the ordered group $\Gal({^*\Reel})$ when $\alpha
\rightarrow +\infty$ and since $k_\alpha$ is
 a real number, we deduce that $D(E,E')=+\infty$.
Thus, $\psi_\infty$ is a morphism
 $(Y,d,D) \rightarrow (X_\infty,\delta_\infty,\Delta_\infty)$ in the
 category of generalized galactic spaces.
\end{proof}

\vskip 3mm

\textbf{5.4} We may think there is a link between the concept of Gromov-Hausdorff convergence and our notion on nonstandard scaling of a metric space. Firstly, let us recall what is the Gromov-Hausdorff distance \cite{GromLafPans}. Given
two subsets $A$ and $B$ of a metric space $(Z,\delta)$, the
Hausdorff distance between $A$ and $B$ in $Z$ is
$$d_H^Z(A,B):=\inf \{\Eps \in \Reel_+^* \ ; \ \text{$A \subset
V_\Eps(B)$ and $B \subset V_\Eps(A)$}\}$$ where, for each subset $C
\subset Z$, the set $V_\Eps(C)$ is the $\Eps$-neighbourhood $\{x \in
Z \ ; \ \delta(x,C)<\Eps\}$ of $C$. Then, the \emph{Gromov-Hausdorff
distance} $d_{GH}(E,F)$ of two metric spaces $E$ and $F$ is the
infimum of numbers $d_H^Z(i(E),j(F))$ for any $(Z,i,j)$ such that
$Z$ is a metric space, $i:E \rightarrow Z$ and $j:F \rightarrow Z$
are isometric embeddings.

The Gromov-Hausdorff distance is not really a distance, mainly
because there are non isometric metric spaces $E$ and $F$ such that
$d_{GH}(E,F)=0$ (for instance $\Reel$ and $\Rationnel$). This
problem disappears in the collection of isometric classes of compact
metric spaces. Nevertheless, we say that a sequence $(E_n)$ of
metric spaces (compact or not compact) converges toward a metric
space $F$ for the Gromov-Hausdorff distance if $d_{GH}(E_n,F)$ converges
to 0 in $\Reel_+$ when $n \rightarrow +\infty$.

\begin{theoreme} Let $(\lambda_n)$ be a standard sequence of strictly positive real numbers such
that  $\lim_{n \rightarrow +\infty}
\lambda_n=0$. If the sequence $(X,\lambda_n d)$ converges to a metric space $(F,d_F)$ for the  Gromov-Hausdorff distance, then, for each infinitely large $\nu
\in {^*\Naturel}$, the $\lambda_\nu$-scaling $X_{\lambda_\nu}$ of $X$ is a galactic
space isometric to the $1$-scaling $F_1$ of $(F,d_F)$.
\end{theoreme}
\begin{proof}
Firstly, we give a more convenient formulation of the Gromov-Hausdorff distance $d_{GH}(E,F)$ of two metric spaces $(E,d_E)$ and $(F,d_F)$: it
is the infimum of the set of real numbers $\Eps>0$ such that there
exists a map $\delta:E \times F \rightarrow \Reel_+$ checking the
two following properties:
\begin{enumerate}
\item the map $d:(E \amalg F)^2 \rightarrow \Reel_+$ defined by
$$d(x,y)=
 \left\{
  \begin{array}{ll}
    d_E(x,y) & \text{if $(x,y) \in E^2$} \\
    d_F(x,y) & \text{if $(x,y) \in F^2$} \\
    \delta(x,y) & \text{if $(x,y) \in E \times F$} \\
    \delta(y,x) & \text{if $(x,y) \in F \times E$} \\
  \end{array}
 \right.
$$ is such that $d(x,z) \leq d(x,y)+d(y,z)$ for every $x$, $y$ and $z$
in the disjoint union $E \amalg F$ of $E$ and $F$,
\item $\text{$(\forall x \in E \ \exists y \in F \ \ \delta(x,y) <
\Eps)$ and $(\forall x \in F \ \exists y \in E \ \ \delta(y,x) <
\Eps$)}.$
\end{enumerate}

Now we return to the proof of the theorem.
From the convergence hypothesis, we deduce that, there is a sequence
$(\Eps_n)$ of real numbers such that $lim_{n \rightarrow
+\infty}\Eps_n=0$ and there is a sequence $(\delta_n)$ of maps from
$X \times F$ to $\Reel_+$ such that, for all $n \in \Naturel$:
\begin{enumerate}
\item the map $d_n:(X \amalg F)^2 \rightarrow \Reel_+$ defined by
$$d_n(x,y)=
 \left\{
  \begin{array}{ll}
    \lambda_n d(x,y) & \text{si $(x,y) \in X^2$} \\
    d_F(x,y) & \text{si $(x,y) \in F^2$} \\
    \delta_n(x,y) & \text{si $(x,y) \in X \times F$} \\
    \delta_n(y,x) & \text{si $(x,y) \in F \times X$} \\
  \end{array}
 \right.
$$ is such that $d_n(x,z) \leq d_n(x,y)+d_n(y,z)$ for every $x$, $y$ and $z$
in the disjoint union $X \amalg F$ of $X$ and $F$,
\item $\text{$(\forall x \in X \ \exists y \in F \ \ \delta_n(x,y) <
\Eps_n)$ and $(\forall x \in F \ \exists y \in X \ \ \delta_n(y,x) <
\Eps_n$)}$.
\end{enumerate}

Let $\nu$ be an infinitely large element of ${^*\Naturel}$. Thus, we
get a map $\delta_\nu:{^*X} \times {^*F} \rightarrow {^*\Reel}$ and
a map $d_\nu:({^*X} \amalg {^*F})^2 \rightarrow {^*\Reel}$ such that
\begin{enumerate}
\item $\forall (x,y) \in ({^*X} \amalg {^*F})^2 \ \
d_\nu(x,y)=
 \left\{
  \begin{array}{ll}
    \lambda_\nu d(x,y) & \text{si $(x,y) \in {^*X}^2$} \\
    d_F(x,y) & \text{si $(x,y) \in {^*F}^2$} \\
    \delta_\nu(x,y) & \text{si $(x,y) \in {^*X} \times {^*F}$} \\
    \delta_\nu(y,x) & \text{si $(x,y) \in {^*F} \times {^*X}$} \\
  \end{array}
 \right.$
\item  $\forall x,y,z \in {^*X} \amalg {^*F} \ \ d_\nu(x,z) \leq
d_\nu(x,y)+d_\nu(y,z)$,
\item $\text{$(\forall x \in {^*X} \ \exists y \in {^*F} \ \ \delta_\nu(x,y) \simeq 0)$
and $(\forall x \in {^*F} \ \exists y \in {^*X} \ \ \delta_\nu(y,x)
\simeq 0)$}$.
\end{enumerate}
Then, we consider the quotient set $G:={^*X} \amalg {^*F}/\sim$ for
the equivalence relation
$$\forall x,y \in {^*X} \amalg {^*F}\ (x \sim y \Longleftrightarrow d_\nu(x,y)
\simeq 0).$$ If, for each $x \in {^*X} \amalg {^*F}$, we denote by
$[x]$ the equivalence class of $x$, we can define a generalized
distance $\delta$ on $G$ such that
$$\forall (x,y) \in ({^*X} \amalg {^*F})^2 \ \
\delta([x],[y])=\text{st}(d_\nu(x,y))$$ and a galactic distance
$\Delta$ on the set $\mathcal{M}_G$ of metric components of $G$ for
$\delta$ such that
$$\forall (C,D) \in \mathcal{M}_G^2 \ \
\Delta(C,D)=\Gal(d_\nu(x_C,x_D))$$ where $x_C,x_D$ are any points in
${^*X} \amalg {^*F}$ verifying  $[x_C] \in C$ and $[x_D] \in D$.
Then, $(G,\delta,\Delta)$ is a galactic space and the map
$$
 \left\{
  \begin{array}{ccc}
    {^*X} & \longrightarrow & G \\
    x & \longmapsto & [x] \\
  \end{array}
 \right.
$$
induces an isometry between the $\lambda_\nu$-scaling of $(X,d)$
and $(G,\delta,\Delta)$. In the same way, the map
$$
 \left\{
  \begin{array}{ccc}
    {^*F} & \longrightarrow & G \\
    x & \longmapsto & [x] \\
  \end{array}
 \right.
$$
induces an isometry between the 1-scaling of $(F,d_F)$ and
$(G,\delta,\Delta)$.

\end{proof}

\vskip 2mm


\nocite{*}

\bibliographystyle{plain}

\bibliography{Contr}


\vskip 1cm

\end{document}